\documentclass[11pt,b5paper,twoside, headrule]{amsart}
\usepackage{amsfonts, amsmath, amssymb,latexsym}
\usepackage{epsfig}
\usepackage[curve]{xy}

\footskip=10mm\parskip 0.2cm\addtocounter{page}{0}
\newtheorem{thm}{Theorem}[section]
\newtheorem{cor}[thm]{Corollary}
\newtheorem{lem}[thm]{Lemma}
\newtheorem{prop}[thm]{Proposition}
\theoremstyle{definition}
\newtheorem{defn}[thm]{Definition}
\newtheorem{examples}[thm]{Examples}
\theoremstyle{remark}
\newtheorem{rem}[thm]{Remark}
\numberwithin{equation}{section}

\title{Quadratic modular symbols on Shimura curves}

\author[P. Bayer-I. Blanco-Chac\'{o}n]{Pilar Bayer and Iv\'{a}n Blanco-Chac\'{o}n}
\thanks{Partially supported by project MTM2009-07024 of the Spanish \textit{Ministerio de Ciencia e Innovaci\'{o}n}.}
\address{Facultad de Matem\'{a}ticas,
Universidad de Barcelona,
Gran Via de les Corts Catalanes, 585.
08007 Barcelona,
Spain.
}

\begin{document}
\renewcommand\baselinestretch{1.2}
\renewcommand{\arraystretch}{1}
\def\base{\baselineskip}
\font\tenhtxt=eufm10 scaled \magstep0 \font\tenBbb=msbm10 scaled
\magstep0 \font\tenrm=cmr10 scaled \magstep0 \font\tenbf=cmb10
scaled \magstep0


\def\evenhead{{\protect\centerline{\textsl{\large{I. Blanco-Chacon}}}\hfill}}

\def\oddhead{{\protect\centerline{\textsl{\large{Elliptic modular symbols}}}\hfill}}

\pagestyle{myheadings} \markboth{\evenhead}{\oddhead}

\thispagestyle{empty}

\begin{abstract}
We introduce the concept of \emph{quadratic} modular symbol and study how these symbols are related to
\emph{quadratic} $p$-adic $L$-functions.
These objects were introduced in \cite{blancobayer} in the case of modular curves.
In this paper, we discuss a method to attach quadratic modular symbols and quadratic $p$-adic $L$-functions
to more general Shimura curves.
\end{abstract}

\maketitle

\section*{Introduction}

To a newform of even weight $2$ for a congruence subgroup $\Gamma_0(N)$,
Mazur-Tate-Teitelbaum \cite{mtt} associated $p$-adic distributions.
The main tools in their construction were the modular integrals along geodesics joining two cusps.
The Mellin-Mazur transforms of these $p$-adic distributions are known as cyclotomic $p$-adic $L$-functions attached to $f$.
They interpolate special values of the complex $L$-function $L(f, s)$.

Alternatively, it is possible to associate different $p$-adic distributions to a newform $f$.
In \cite{blancobayer}, we have recently constructed quadratic $p$-adic $L$-functions from integration
along geodesics connecting two quadratic imaginary points of the complex upper half-plane $\mathcal{H}$.
Our construction presents notable differences with respect to the cyclotomic $p$-adic $L$-functions.
For instance, the images of the cyclotomic $p$-adic $L$-functions lie on finite dimensional $p$-adic vector spaces
whereas the images of the quadratic $p$-adic $L$-functions lie on infinite dimensional $p$-adic Banach spaces.
More crucial is the fact that if the newform $f$ has rational coefficients,
then the quadratic $p$-adic $L$-functions produce algebraic points on the corresponding elliptic curve $E_f$,
instead of interpolating special values of the function $L(E_f, s)$.

For a cocompact arithmetic Fuchsian group of the first kind,
the absence of cusps prevents the definition of $p$-adic $L$-functions via modular integrals connecting rational numbers.
In this paper, we extend the concept of quadratic modular symbols as defined in \cite{blancobayer}
to cover the cocompact case, and relate them to quadratic $p$-adic $L$-functions.
Classical modular symbols span a finite dimensional $\mathbb{C}$-vector space;
on the contrary, quadratic modular symbols do not.
This situation reflects the fact that quadratic $p$-adic $L$-functions
take values in an infinite dimensional $p$-adic Banach space.

In section 1, we present some facts about quaternion algebras and Shimura curves;
we refer the reader to \cite{alsinabayer} and \cite{vigneras} for more detailed discussions.
In section 2, we review some results on the homology of Shimura curves.
In section 3, we introduce the quadratic modular symbols and compare them with the classical modular ones.
In section 4, we compare the two constructions of the $p$-adic $L$-functions:
that of Mazur-Tate-Teitelbaum, and the new one. We conclude the section with the construction of a $p$-adic $L$-function for certain Shimura curves.

We are grateful to Federico Cantero for providing us with reference \cite{armstrong},
and to Ignasi Mundet for helpful discussion about the topology of orbifolds. Finally, we thank the Mathematics Department of Euskal Herria University
for giving us the possibility to publish this paper in the proceedings of the
\emph{Fourth Meeting on Number Theory}.

\tableofcontents
\section{Quaternion algebras and Shimura curves}

\subsection{Quaternion algebras and quaternion orders}

Let $a,b\in\mathbb{Z}$ and let
$H=\left(\frac{a,b}{\mathbb{Q}}\right)$
be the quaternion $\mathbb{Q}$-algebra generated by $I$ and $J$
with the standard relations
$I^2=a,J^2=b,IJ=-JI$.
Denote $K=IJ$.
The reduced trace and the reduced norm of a quaternion $\omega= x+yI+zJ+tK\in H$ are defined by
$$
\mathrm{Tr}(\omega)=\omega+\overline{\omega}=2x, \quad
\mathrm{N}(\omega)=\omega\overline{\omega}=x^2-ay^2-bz^2+abt,
$$
where $\overline{\omega}=x-yI-zJ-tK$ denotes the conjugate of $\omega$.
The following map yields an isomorphism of quaternion $\mathbb{Q}$-algebras
$$
\begin{array}{ccc}
\phi: \left(\dfrac{a,\, b}{\mathbb{Q}}\right) & \to & \mathrm{M}\left(2,\mathbb{Q}(\sqrt{a})\right)\\
x+yI+zJ+tK & \mapsto &
\left(\begin{array}{ccc} x+y\sqrt{a} &\phantom{x} & z+t\sqrt{a}\\
b(z-t\sqrt{a})&\phantom{x} & x-y\sqrt{a}\end{array}\right).
\end{array}
$$
Notice that for any $\omega\in H$,
$\mathrm{N}(\omega)=\mathrm{det}\left(\phi(\omega)\right)$, and $\mathrm{Tr}(\omega)=\mathrm{Tr}\left(\phi(\omega)\right)$.

For any place $p$ of $\mathbb{Q}$ (possibly including $p=\infty$),
$H_p:=H\otimes_{\mathbb{Q}}\mathbb{Q}_p$ is a quaternion $\mathbb{Q}_p$-algebra.
If $H_p$ is a division algebra, it is said that $H$ is ramified at $p$.
As is well known, the quaternion algebra $H$ is ramified at a finite even number of places.
The discriminant $D_H$ is defined as the product of the primes at which $H$ ramifies.
Moreover, two quaternion $\mathbb{Q}$-algebras are isomorphic if and only if they have the same discriminant.

\begin{defn}
Let $H$ be a quaternion $\mathbb{Q}$-algebra.
If $D_H=1$, $H$ is said to be non-ramified; in this case, it is isomorphic to $M(2, \mathbb{Q})$.
If $H$ is ramified at $p=\infty$, it is said to be definite, and indefinite otherwise.
An indefinite quaternion algebra is said to be small ramified if $D_H$ is equal to the product of two distinct primes.
\end{defn}

In this paper we shall deal with indefinite quaternion $\mathbb{Q}$-algebras.
The following result gives a useful presentation of non-ramified and small ramified quaternion $\mathbb{Q}$-algebras.

\begin{thm}[Alsina-Bayer \cite{alsinabayer}]\label{structure}
Let $H=\left(\dfrac{a,b}{\mathbb{Q}}\right)$ be a quaternion $\mathbb{Q}$-algebra.
\begin{itemize}
\item[(i)]
If $D_H=1$, then $H\simeq\mathrm{M}\left(2,\mathbb{Q}\right)\simeq \left(\dfrac{1,-1}{\mathbb{Q}}\right)$.
\item[(ii)]
If $D_H=2p$, $p$ prime, $p\equiv 3\pmod{4}$, then $H\simeq\left(\dfrac{p,-1}{\mathbb{Q}}\right)$.
\item[(iii)]
If $D_H=pq$, $p,q$ primes, $q\equiv 1\pmod{4}$ and $\left(\dfrac{p}{q}\right)=1$,
then $H\simeq\left(\dfrac{p,q}{\mathbb{Q}}\right)$.
\end{itemize}
If $a$ and $b$ are prime numbers, then $H$ satisfies one, and only one, of the three previous statements.
\end{thm}

Let $H$ be a quaternion $\mathbb{Q}$-algebra.
An element $\alpha\in H$ is said to be integral if $\mathrm{N}(\alpha),\mathrm{Tr}(\alpha)\in \mathbb{Z}$.
A $\mathbb{Z}$-lattice $\Lambda$ of $H$ is a finitely generated torsion free $\mathbb{Z}$-module contained in $H$.
A $\mathbb{Z}$-ideal of $H$ is a $\mathbb{Z}$-lattice $\Lambda$ such that $\mathbb{Q}\otimes \Lambda\simeq H$.
A $\mathbb{Z}$-ideal is not in general a ring.
An order $\mathcal{O}$ of $H$ is a $\mathbb{Z}$-ideal which is a ring.
Each order of a quaternion algebra is contained in a maximal order.
In an indefinite quaternion algebra, all the maximal orders are conjugate (cf.\,\cite{vigneras}).
An Eichler order is the intersection of two maximal orders.

Let $H=\left(\dfrac{a,b}{\mathbb{Q}}\right)$ be an indefinite quaternion $\mathbb{Q}$-algebra.
Given a maximal order $\mathcal{O}_{H}$,
denote by $\mathcal{O}_H^1$ the multiplicative group
of elements of $\mathcal{O}_H$ of reduced norm equal to $1$, and let $\Gamma_H^1$
be its image under $\phi$.
A Fuchsian group of the first kind $\Gamma \subseteq \mathrm{GL}\left(2,\mathbb{R}\right)$
is called arithmetic if it is commensurable with $\Gamma_H^1$ for some quaternion algebra $H$.

\begin{prop}[cf.\,\cite{vigneras}]
Let $\mathcal{O}$ be an Eichler order of $H$. Then $\mathcal{O}_p=\mathcal{O}\otimes\mathbb{Z}_p$ is a
$\mathbb{Z}_p$-order of $H_p$.
Moreover, there exists a unique $n\geq 0$ such that $\mathcal{O}_p$ is conjugated to the Eichler order
$$
\mathcal{O}_n=\left\{\left(\begin{array}{ccc}a & \phantom{x} & b\\cp^n & \phantom{x} &
d\end{array}\right);\,\,a,b,c,d\in\mathbb{Z}_p\right\}.
$$
The level of the local Eichler order is defined as $p^n$.
\end{prop}

To define the global level of $\mathcal{O}$, write
$\mathcal{O}=\mathcal{O}'\cap\mathcal{O}^{''}$ with $\mathcal{O}',\mathcal{O}^{''}$ maximal orders and tensor by
$\mathbb{Z}_p$. The global level is then the product of all local levels.

\begin{prop}[Alsina-Bayer \cite{alsinabayer}] Let $N\geq 1$ and $p,q$ be different primes as in Theorem \ref{structure}.
\begin{itemize}
\item[(i)]
$\mathcal{O}_0(1,N)=\left\{\left(\begin{array}{ccc}a & \phantom{x} & b\\cN & \phantom{x}
& d\end{array}\right);\,\,a,b,c,d\in\mathbb{Z}\right\}$ is an Eichler order of level $N$ in $\mathrm{M}\left(2,\mathbb{Q}\right)$.
\item[(ii)]
$\mathcal{O}_M(1,N)=\mathbb{Z}+\mathbb{Z}(J+K)/2+\mathbb{Z}N(-J+K)/2+\mathbb{Z}(1-I)/2$
is an Eichler order of level $N$ in $M=\left(\frac{1,-1}{\mathbb{Q}}\right)$, the matrix algebra.
\item[(iii)]
If $D=2p$, $N|(p-1)/2$ and $N$ is square free, then
$\mathcal{O}(2p,N)=\mathbb{Z}+\mathbb{Z}I+\mathbb{Z}NJ+\mathbb{Z}\left(\frac{1+I+J+K}{2}\right)$
is an Eichler order of level $N$ in $\left(\frac{p,-1}{\mathbb{Q}}\right)$,
for $N\vert(p-1)/2$, $N$ square free.
\item[(iv)]
If $D=pq$, $N|(q-1)/4$, $(N,p)=1$ and $N$ is square free,
then $\mathbb{Z}+\mathbb{Z}NI+\mathbb{Z}(1+J)/2+\mathbb{Z}(I+K)/2$ is an Eichler order of level $N$ in
$\left(\frac{p, q}{\mathbb{Q}}\right)$, for $N\vert(p-1)/4$, $p\nmid N$, $N$ square free.
\end{itemize}
\label{eichlerprop}
\end{prop}

\begin{rem}
For $D=\pm 1,2p,pq$, with $p,q$ primes as in \ref{structure}, and $N\geq 1$, denote by $\Gamma(D,N)$
the image under $\phi$ of the group of units of reduced norm $1$ in the Eichler
orders given in Proposition \ref{eichlerprop}.
The groups $\Gamma(D,N)$ are arithmetic Fuchsian groups of the first kind.
In particular, $\Gamma(1,N)=\Gamma_0(N)$.
\label{comentmodulareichler}
\end{rem}

\subsection{Shimura curves}
Denote by $\mathcal{H}$ the complex upper half-plane endowed with the hyperbolic metric
$$
\delta(z_1,z_2)=\left|\mathrm{arccosh}\left(1+\frac{|z_1-z_2|^2}{2\mathrm{Im}(z_1)\mathrm{Im}(z_2)}\right)\right|.
$$
The hyperbolic lines are the semilines which are orthogonal to the real axis and the semicircles centered at real points.

The group $\mathrm{SL}(2, \mathbb{R})$ acts on $\mathcal{H}$ by M\"{o}bius transformations
and its action factorizes through
$\mathrm{PSL}(2, \mathbb{R})$.

\begin{defn}Let $\gamma\in\mathrm{SL}\left(2,\mathbb{R}\right)$, $\gamma\not= \pm \mathrm{Id}$. Then
\begin{itemize}
\item[(a)]
$\gamma$ is elliptic if it has a fixed point $z\in\mathcal{H}$, and the other fixed point is $\overline{z}$.
\item[(b)]
$\gamma$ is parabolic if it has a unique fixed point in $\mathbb{R}\cup\{i\infty\}$.
\item[(c)]
$\gamma$ is hyperbolic if it has two distinct fixed points in $\mathbb{R}\cup\{i\infty\}$.
\end{itemize}
\end{defn}

\begin{prop}
Let $\gamma\in\Gamma\subseteq \mathrm{SL}\left(2,\mathbb{R}\right)$, $\gamma\not= \pm \mathrm{Id}$.
Then, $\gamma$ is elliptic if and only if $|\mathrm{Tr}(\gamma)|<2$.
If $\mathrm{Tr}(\gamma)=0$, then $\gamma$ has order 2 or 4 depending on $\mathrm{-Id}\in\Gamma$
or $\mathrm{-Id}\not\in\Gamma$.
If $\mathrm{Tr}(\gamma)=1$, then $\gamma$ has order 3 or 4 depending on
$\mathrm{-Id}\in\Gamma$ or $\mathrm{-Id}\not\in\Gamma$.
These are the two only possibilities for elliptic transformations with integral traces.
\end{prop}

\begin{defn}
Let $\Gamma\subseteq\mathrm{SL}\left(2,\mathbb{R}\right)$ be a discrete subgroup acting on $\mathcal{H}$.
A point $z\in\mathcal{H}$ is said to be elliptic if its isotropy group in $\Gamma$
is generated by an elliptic element of $\Gamma$. The order of $z$ is the order of its isotropy group.
\end{defn}

Fix $H=\left(\dfrac{a,b}{\mathbb{Q}}\right)$ be an indefinite quaternion $\mathbb{Q}$-algebra and
$\Gamma^1=\phi\left(\mathcal{O}_H^1\right)$.
Let $\Gamma$ be an arithmetic Fuchsian group of the first kind commensurable with $\Gamma^1$.
It acts on $\mathcal{H}$ by M\"{o}bius transformations
and the action factors through its image in $\mathrm{PSL}\left(2,\mathbb{R}\right)$.
The quotient $\Gamma\backslash\mathcal{H}$ has an analytic structure of Riemann surface.
If this Riemann surface is compact, then $\Gamma$ is said to be cocompact.
The Riemann surface $\Gamma\backslash\mathcal{H}$ is analytically isomorphic to an open subset
of a smooth algebraic curve defined over $\mathbb{Q}$,
which is denoted $X\left(\Gamma\right)$ (cf.\,\cite{shimura1967}).

Notice that the order of an elliptic point in $\Gamma\backslash\mathcal{H}$ can only be 2 or 3.

The following result is well known.

\begin{thm} Let $\Gamma$ be an arithmetic Fuchsian group of the first kind commensurable with $\Gamma^1_{H}$. Then $\Gamma$ is cocompact if and only if $H$ is a division algebra.
\end{thm}
\begin{examples}
The quaternion $\mathbb{Q}$-algebra with discriminant 1 is isomorphic to $\mathrm{M}\left(2,\mathbb{Q}\right)$ and the maximal order is $\mathrm{M}\left(2,\mathbb{Z}\right)$. As seen in Remark \ref{comentmodulareichler}, the congruence subgroup $\Gamma_0(N)$ is provided by the Eichler order $\mathcal{O}_0(1,N)$. Thus, $\Gamma_0(N)$ is not cocompact. The Riemann surface $\Gamma_0(N)\backslash\mathcal{H}$ becomes compact by adding the set of cusps, $\mathrm{SL}\left(2,\mathbb{Z}\right)i\infty$. The compact Riemann surface corresponds to $X_0(N)\left(\mathbb{C}\right)$, the set of complex points of the modular curve $X_0(N)$.
\end{examples}

\begin{examples}
Let $H=\left(\dfrac{a,b}{\mathbb{Q}}\right)$ be a quaternion $\mathbb{Q}$-algebra of discriminant $D>1$.
The group $\Gamma(D,N)$ is cocompact because it does not have parabolic elements.
The Riemann surface $\Gamma(D,N)\backslash\mathcal{H}$ is compact and analytically isomorphic to an algebraic curve
$X(D,N)$ (cf.\,\cite{shimura1967}).
Fundamental domains for several $\Gamma(D,1)$ can be consulted at \cite{alsinabayer}.
\end{examples}

Let $\mathrm{GL}\left(2,\mathbb{R}\right)^{+}$ be the multiplicative subgroup of real matrices with positive discriminant and
$\gamma=\begin{pmatrix}a & \phantom{x}& b\\ c & \phantom{x}& d\end{pmatrix}
\in\mathrm{GL}\left(2,\mathbb{R}\right)^{+}.
$
Define $\rho(\gamma,z)=\dfrac{\mathrm{det}(\gamma)^{1/2}}{cz+d}$. Let $f:\mathcal{H}\to\mathbb{C}$ be a holomorphic function. Denote
$$f|_k\gamma(z)=\rho(\gamma,z)^kf(\gamma(z)).$$
\begin{defn}
An automorphic form of weight $k$ for a cocompact group $\Gamma$ is a holomorphic function $f$ on $\mathcal{H}$ such that
$f|_k\gamma=f$, for any $\gamma\in\Gamma$. The $\mathbb{C}$-vector space of automorphic forms of weight $k$ for $\Gamma$ is denoted by $S_k\left(\Gamma\right)$.
\end{defn}
From now on, we will restrict ourselves to automorphic forms of weight 2.

Let us denote by $\Omega$ the sheaf of holomorphic differentials on $X\left(\Gamma\right)$ and let $g$ denote the genus of $X\left(\Gamma\right)$. One has an isomorphism
$$
S_2\left(\Gamma\right)\simeq H^0\left(X\left(\Gamma\right),\Omega\right).
$$
Hence, in particular, the dimension of $S_2\left(\Gamma\right)$ as a $\mathbb{C}$-vector space is $g$.

\section{The homology of Shimura curves}
\subsection{The structure of the homology of a Shimura curve}
Let $\Gamma$ be an arithmetic Fuchsian group of the first kind attached to an indefinite quaternion $\mathbb{Q}$-algebra of discriminant $D>1$. The corresponding Shimura curve $X\left(\Gamma\right)$ is compact. The homology group $H_1\left(X\left(\Gamma\right)\left(\mathbb{C}\right),\mathbb{R}\right)$ contains the maximal lattice $H_1\left(X\left(\Gamma\right)\left(\mathbb{C}\right),\mathbb{Z}\right)$. We will use the following result to study the structure of $H_1\left(X\left(\Gamma\right)\left(\mathbb{C}\right),\mathbb{R}\right)$.
\begin{thm}[Armstrong \cite{armstrong}] Let $\Gamma$ be a group which acts simplicially on a simplicial complex $U$. Let $E$ be the normal subgroup of $\Gamma$ of elements with a fixed element of $U$. Given $\alpha\in U$ and $g\in\Gamma$, define $\phi_{\alpha}(g)$ as the homotopy class of an edge path joining $\alpha$ with $g(\alpha)$. Then the map $\phi_{\alpha}$ factors by a map $f_{\alpha}:\Gamma/E\to\pi_1\left(\Gamma\backslash U,\alpha\right)$, which is an isomorphism.
\label{thmarmstrong}
\end{thm}

Notice that M\"{o}bius transforms are conform; hence, they preserve geodesic triangles. In particular, arithmetic Fuchsian groups act simplicially on $\mathcal{H}$.

\begin{thm}Let $\Gamma$ be an arithmetic Fuchsian group of the first kind. Denote by $E$ and $P$ the sets of elliptic and parabolic elements of $\Gamma$, respectively. Let $\Gamma'$ be the commutator of $\Gamma$. Let $\alpha\in\mathcal{H}$ if $\Gamma$ is cocompact or $\alpha\in\mathcal{H}\cup\mathrm{P}^1\left(\mathbb{Q}\right)$, otherwise. For any $g\in\Gamma$, define $\phi_{\alpha}(g)=\{\alpha,g(\alpha)\}\in H_1\left(X\left(\Gamma\right)\left(\mathbb{C}\right),\mathbb{Z}\right)$. Then
\begin{itemize}
\item[(i)] If $\Gamma$ is not cocompact, then, for any $\alpha\in\mathcal{H}\cup\mathrm{P}^1\left(\mathbb{Q}\right)$, there is an exact sequence of groups
$$
0\to \Gamma'EP \to\Gamma\buildrel\phi_{\alpha}\over\rightarrow H_1\left(X\left(\Gamma\right)(\mathbb{C}),\mathbb{Z}\right)\to 0.
$$
\item[(ii)] If $\Gamma$ is cocompact, then, for any $\alpha\in\mathcal{H}$, there is an exact sequence of groups
$$
0\to \Gamma'E \to\Gamma\buildrel\phi_{\alpha}\over\rightarrow H_1\left(X\left(\Gamma\right)(\mathbb{C}),\mathbb{Z}\right)\to 0.
$$
\end{itemize}
In both cases, the map $\phi_{\alpha}$ is independent of $\alpha$.
\label{key1}
\end{thm}
\begin{proof}
We prove the result for cocompact $\Gamma$. The other case is proved in \cite{manin}. First, we check that for any $\alpha\in\mathcal{H}$, $\phi_{\alpha}$ is a group homomorphism. Thus, take $g,h\in\Gamma$. Since $\mathcal{H}$ is simply connected, we have
$$\{\alpha,gh(\alpha)\}=\{\alpha,g(\alpha)\}+\{g(\alpha),gh(\alpha)\}=\{\alpha,g(\alpha)\}+\{\alpha,h(\alpha)\},$$
hence, $\phi_{\alpha}$ is a group homomorphism.

Next, we check the independency on $\alpha$. Let $\alpha,\beta\in\mathcal{H}$. We can decompose
$$
\phi_{\alpha}(g)=\{\alpha,\beta\}+\{\beta,g(\beta)\}+\{g(\beta),g(\alpha)\}=\{\alpha,\beta\}+\{\beta,g(\beta)\}+\{\beta,\alpha\}=\phi_{\beta}(g).
$$
We claim that the commutator subgroup of $\Gamma/E$ is $\Gamma'E/E$. To see this, let us consider the projection $p:\Gamma\to\Gamma/E$, which sends $\Gamma'$ onto $\Gamma'E/E$. Hence, it induces a projection $\overline{p}:\Gamma/\Gamma'\to\Gamma/\Gamma'E$, so that $\Gamma/\Gamma'E\simeq\left(\Gamma/E\right)/\left(\Gamma'E/E\right)$ is abelian. Thus, $\left(\Gamma/E\right)'\subseteq \Gamma'E/E$. On the other hand, for any $a,b\in\Gamma$, one has, by normality of $E$, that $aba^{-1}b^{-1}E=aEbE(aE)^{-1}(bE)^{-1}$. This shows the reverse inclusion. In particular, there is an isomorphism $\psi: H_1\left(X\left(\Gamma\right)\left(\mathbb{C}\right),\mathbb{Z}\right)\to \Gamma/\Gamma'E$.

Consider the commutative diagram
$$
\begin{array}{cccc}
\Gamma & \buildrel\phi_{\alpha}\over\longrightarrow & H_1\left(X\left(\Gamma\right)\left(\mathbb{C}\right),\mathbb{Z}\right) &\\
\downarrow & & \downarrow \psi &\\
\Gamma/E & \longrightarrow & \left(\Gamma/E\right)/\left(\Gamma'E/E\right) & \longrightarrow 0.
\end{array}
$$
It follows that $\mathrm{Ker}(\phi_{\alpha})=\Gamma'E$; thus, the result follows.

\end{proof}

Let $G$ be a set of generators of $\Gamma$. If $\Gamma$ is cocompact, denote by $H(G)$ the set of elements of $G$ which are neither elliptic nor commutators. If $\Gamma$ is not cocompact, we define $H(G)$ as the set of elements of $G$ which are neither elliptic nor parabolic, nor commutators. Let $i\in\mathcal{H}$ denote the imaginary unit.

\begin{defn} A family of quadratic distinguished classes is
$$
\{i,\sigma(i)\}_{\sigma\in H(G)},
$$
with $G$ a set of generators of $\Gamma$.
\end{defn}

We have the following
\begin{prop}Any class in $H_1\left(X\left(\Gamma\right)(\mathbb{C}),\mathbb{Z}\right)$ can be represented as a sum of quadratic distinguished classes. In particular, the quadratic distinguished classes generate $H_1\left(X\left(\Gamma\right)(\mathbb{C}),\mathbb{R}\right)$ as a real vector space.
\label{generation}
\end{prop}
\begin{proof}
By Theorem \ref{key1}, any class in $H_1\left(X\left(\Gamma\right)(\mathbb{C}),\mathbb{Z}\right)$ is of the form $\{i,g(i)\}$ for some $g\in\Gamma$. Decompose $g=\eta_1\cdots\eta_l$ where the $\eta_k$ are generators of $\Gamma$. Notice that
$$
\{i,g(i)\}=\{i,\eta_1(i)\}+\{\eta_1(i),g(i)\}=\{i,\eta_1(i)\}+\{i,\eta_2\eta_3\cdots\eta_l(i)\}.
$$

Iterating, we have
$$
\{i,g(i)\}=\sum_{k=1}^{l}\{i,\eta_k(i)\}.
$$
If $\eta_k\not\in H(G)$, we can consider its fixed point $\tau\in\mathcal{H}$ and have
$$\{i,\eta_k(i)\}=\{i,\tau\}+\{\tau,\eta_k(i)\}=\{i,\tau\}+\{\eta_k(\tau),\eta_k(i)\}=\{i,\tau\}+\{\tau,i\}=0.$$
\end{proof}

\begin{rem}
For any arithmetic Fuchsian group of the first kind $\Gamma$, it is possible to find a set of generators of $\Gamma$ of the form $\{\eta_1,...,\eta_{2g},\varepsilon_1,...,\varepsilon_t\}$ with $\eta_k$ hyperbolic and $\varepsilon_j$ elliptic, parabolic or commutator for any $k,j$ (cf.\,\cite{takeuchi}). Thus, the set $\{\{i,\eta_1(i)\},...,\{i,\eta_{2g}(i)\}\}$ is a basis of $H_1\left(X\left(\Gamma\right)\left(\mathbb{C}\right),\mathbb{Z}\right)$. We illustrate this fact numerically by showing a basis of the homology for several Shimura curves.
\end{rem}

\begin{examples}
Let us define the matrices
$$
V_k=
\begin{pmatrix}
k^{'} & \phantom{x} & 1\\
-(k^{'}k+1) & \phantom{x} & -k
\end{pmatrix}
,
$$
with $1\leq k,k^{'}\leq p-1$ and $kk^{'}\equiv -1\pmod{p}$. Denote the unitary translation by $T$. Table \ref{table:1} shows generators and fundamental relations of $\Gamma_0(N)$ (with $N=p$ prime).
{\tiny
\begin{table}[!h]
\caption{}
\label{table:1}
\begin{tabular}{|c|l|l|c|c}
\hline $p$ & Minimal set of generators of $\Gamma_0(p)$ & Relations & Genus of $X_0(p)$\\
\hline 2   & $T,V_1$ & $V_1^2=1$ & 0\\
\hline 3 & $T,V_2$ & $V_2^3=1$ & 0\\
\hline 5 & $T,V_2,V_3$ & $V_2^2=V_3^3=1$ & 0\\
\hline 7 & $T,V_3,V_5$ & $V_3^3=V_5^3=1$ & 0\\
\hline 11 & $T,V_4,V_6$ & \phantom{x} & 1\\
\hline 13 & $T,V_4,V_5,V_8,V_10$ & $V_5^2=V_8^2=V_4^3=V_10^3=1$ & 0\\
\hline 17 & $T,V_4,V_7,V_9,V_{13}$ & $V_4^2=V_{13}^2=1$ & 1\\
\hline 19 & $T,V_5,V_8,V_{12},V_{13}$ & $V_8^2=V_{12}^2=1$ & 1\\
\hline 23 & $T,V_8,V_{10},V_{12},V_{14}$ & \phantom{x} & 2\\
\hline 29 & $T,V_6,V_{12},V_{13},V_{15},V_{17},V_{22}$ & $V_{12}^2=V_{17}^2=1$ &2\\
\hline 31 & $T,V_6,V_9,V_{13},V_{17},V_{21},V_{26}$ & $V_6^3=V_{26}^3=1$ & 2\\
\hline 37 & $T,V_6,V_8,V_{11},V_{16},V_{20},V_{27},V_{28},V_{31}$ & $V_{11}^3=V_{27}^3=1$ &3\\
\hline 41 & $T,V_7,V_9,V_{16},V_{19},V_{21},V_{24},V_{32},V_{33}$ & $V_9^2=V_{32}^2=1$ & 3\\
\hline 43 & $T,V_7,V_{13},V_{15},V_{18},V_{24},V_{27},V_{29},V_{37}$ & $V_7^3=V_{37}^3=1$ &3\\
\hline 47 & $T,V_{13},V_{16},V_{19},V_{22},V_{24},V_{27},V_{30},V_{33}$ & $\phantom{x}$ & 4\\
\hline 53 & $\begin{array}{l}T,V_{12},V_{14},V_{20},V_{23},V_{25},V_{27},V_{30},V_{32}\\V_{38},V_{40}\end{array}$ & $V_{23}^2=V_{30}^2=1$ & 4\\
\hline 59 & $\begin{array}{l}T,V_{12},V_{15},V_{20},V_{26},V_{28},V_{30},V_{32},V_{38}\\V_{43},V_{46}\end{array}$ & $\phantom{x}$ &5\\
\hline 61 & $\begin{array}{l}T,V_{9},V_{11},V_{14},V_{18},V_{25},V_{28},V_{32},V_{35}\\V_{42},V_{48},V_{50},V_{51}\end{array}$ & $V_{11}^2=V_{50}^2=V_{14}^3=V_{48}^3=1$ &4\\
\hline 67 & $\begin{array}{l}T,V_{10},V_{18},V_{21},V_{24},V_{30},V_{31},V_{35},V_{38}\\V_{42},V_{45},V_{48},V_{56}\end{array}$ & $V_{30}^3=V_{28}^3=1$ & 5\\
\hline 71 & $\begin{array}{l}T,V_{9},V_{13},V_{24},V_{26},V_{28},V_{34},V_{36},V_{42}\\V_{44},V_{46},V_{57},V_{61}\end{array}$ & $\phantom{x}$ &6\\
\hline 73 & $\begin{array}{l}T,V_{9},V_{11},V_{17},V_{22},V_{25},V_{27},V_{33},V_{39}\\V_{46},V_{47},V_{50},V_{55},V_{61},V_{65}\end{array}$ & $V_{27}^2=V_{46}^2=V_9^3=V_{65}^3=1$ &5\\
\hline 79 & $\begin{array}{l}T,V_{12},V_{20},V_{24},V_{25},V_{30},V_{34},V_{36},V_{42}\\V_{44},V_{48},V_{53},V_{56},V_{58},V_{66}\end{array}$ & $V_{24}^3=V_{56}^3=1$ & 6\\
\hline 83 & $\begin{array}{l}T,V_{14},V_{22},V_{28},V_{30},V_{32},V_{37},V_{40},V_{42}\\V_{45},V_{50},V_{52},V_{54},V_{60},V_{68}\end{array}$ & $\phantom{x}$ &7\\
\hline 89 & $\begin{array}{l}T,V_{10},V_{18},V_{21},V_{31},V_{34},V_{36},V_{39},V_{43}\\V_{45},V_{49},V_{52},V_{55},V_{57},V_{67},V_{70},V_{78}\end{array}$ & $V_{34}^2=V_{55}^2=1$ &7\\
\hline 97 & $\begin{array}{l}T,V_{11},V_{15},V_{22},V_{23},V_{28},V_{30},V_{36},V_{40}\\V_{46},V_{50},V_{56},V_{62},V_{66},V_{68},V_{73},V_{75},V_{81},V_{85}\end{array}$ & $V_{22}^2=V_{75}^2=V_{36}^3=V_{62}^3=1$ & 7\\
\hline 101 & $\begin{array}{l}T,V_{10},V_{19},V_{23},V_{27},V_{30},V_{35},V_{40},V_{43}\\V_{49},V_{51},V_{57},V_{60},V_{65},V_{70},V_{73},V_{77},V_{81},V_{91}\end{array}$ & $V_{10}^2=V_{91}^2=1$ &8\\
\hline
\end{tabular}
\end{table}
}
\end{examples}

\begin{examples}The Shimura curve $X(6,1)$ has genus 0 and $\Gamma(6,1)$ can be generated by six elliptic matrices. The Shimura curve $X(10,1)$ has genus 0 and $\Gamma(10,1)$ can be generated by three elliptic matrices. The Shimura curve $X(15,1)$ has genus 1 and $\Gamma(15,1)$ has the following minimal set of generators:
$$
\alpha=\frac{1}{2}\left(\begin{array}{ccc}3 & \phantom{x} & 1 \\ 5 & \phantom{x} & 3\end{array}\right),\,
h=\left(\begin{array}{ccc}2+\sqrt{3} & \phantom{x} & 0\\0 & \phantom{x} & 2-\sqrt{3}\end{array}\right),\,
\beta=\frac{1}{2}\left(\begin{array}{ccc}1+2\sqrt{3} & \phantom{x} & 3-2\sqrt{3}\\15+10\sqrt{3} & \phantom{x} & 1-2\sqrt{3}\end{array}\right).
$$
The matrices $\alpha,h$ are hyperbolic and $\beta$ is elliptic of order 6.
\end{examples}

\subsection{Modular integrals} In what follows, $\Gamma$ will denote an arithmetic Fuchsian group commensurable with $\Gamma^1_H$ for some quaternion $\mathbb{Q}$-algebra $H$.

\begin{defn} Let $f$ be an automorphic form of weight 2 for $\Gamma$ and $\tau\in\mathcal{H}$ a quadratic imaginary point. The modular integral attached to $f$ and $\tau$ is the map
$$
\begin{array}{rccc}
\phi_{f}^{\tau}: &\Gamma^1_H\tau & \longrightarrow & \mathbb{C}\\
& \gamma(\tau) & \mapsto &\displaystyle \int_{\gamma(\tau)}^{\tau}f.
\end{array}
$$
We will denote $\phi_f=\phi_f^i$ when $\tau=i$, the imaginary unit.
\end{defn}
\begin{rem}
If $\Gamma=\Gamma_0(N)$ and instead of a quadratic imaginary point we consider $\tau=i\infty$ we have the classical modular integral as in \cite{mtt}.
\end{rem}
Let $\Gamma$ be an arithmetic Fuchsian group of the first kind commensurable with $\Gamma^1_H$ for some quaternion $\mathbb{Q}$-algebra $H$. In particular, there is a finite number of coset representatives of $(\Gamma^1_H\cap\Gamma)\backslash\Gamma^1_H$.

\begin{lem}Let $f\in S_2\left(\Gamma\right)$ and let $A,\gamma\in\mathrm{GL}\left(2,\mathbb{R}\right)^{+}$. Then
$$
\phi_f^{\tau}(A\gamma(\tau))=\phi_{f|A}^{\tau}(\gamma(\tau))+\phi_f^{\tau}(A(\tau)),
$$
where $\phi_{f|A}^{\tau}(\gamma(\tau))=\int_{\gamma(\tau)}^{\tau}f|A$.
\end{lem}
\begin{proof}
Suppose $\tau=i$. First, notice that $dA(z)=\rho(A,z)^2dz$, hence
$$
\phi_{f|A}(\gamma(i))=\int_{\gamma(i)}^{i}\rho(A,z)^2f(A(z))dz=\int_{A\gamma(i)}^{A(i)}f(w)dw.
$$
Since $f$ is holomorphic in the upper half-plane, the integral along the triangle with vertices $A(i)$, $A\gamma(i)$ and $i$ vanishes. Hence
$$
\phi_{f|A}(\gamma(i))=\int_{A\gamma(i)}^{i}f(z)dz+\int_{i}^{A(i)}f(z)dz,
$$
and the result holds.
\end{proof}

As a corollary we have the following
\begin{prop}The $\mathbb{Z}$-module $\Sigma_f=\langle\phi_f(\gamma(i));\,\gamma\in\Gamma^1_H\rangle_{\mathbb{Z}}$ is finitely generated and torsion-free. Given $G$ a minimal set of generators of $\Gamma$, then
$$\mathrm{rk}\left(\Sigma_f\right)\leq \mathrm{card}\left(G\right)+\left[\Gamma^1_H:\Gamma\cap\Gamma^1_H\right].$$
\label{finitude}
\end{prop}
\begin{proof}
Let $\left\{A_l\right\}_{1\leq l\leq n}$ be a set of coset representatives of $(\Gamma\cap\Gamma^1_H)\backslash\Gamma^1_H$ and let $G=\left\{B_j\right\}_{1\leq j\leq m}$ be a minimal set of generators of $\Gamma$.

Let $A\in\Gamma^1_H$. There exists $B\in\Gamma\cap\Gamma^1_H$ and $l_0\in\{1,...,n\}$ such that $A=BA_{l_0}$. Hence
$$
\phi_f(BA_{l_0}(i))=\phi_{f|B}(A_{l_0}(i))+\phi_{f}(B(i))=\phi_{f}(A_{l_0}(i))+\phi_f(B(i)).
$$
Now, write $B=B_{j_1}...B_{j_r}$ with $\{j_1,...,j_r\}\subseteq \{1,...,m\}$. Hence
$$\phi_{f}(B(i))=\phi(B_{j_2}...B_{j_r}(i))+\phi_f(B_{j_1}(i)),$$
and
$$
\phi_f(A(i))\in\langle\phi_f(A_l(i)),\phi_f(B_j(i));\,1\leq l\leq n,1\leq j\leq m\rangle_{\mathbb{Z}}.
$$
\end{proof}

\section{Modular symbols}

\subsection{Classical modular symbols}
We begin by recalling the following
\begin{defn}[Manin \cite{manin}, Pollack-Stevens \cite{pollack}] Denote $\Delta_0=\mathrm{SL}\left(2,\mathbb{Z}\right)i\infty$. Let $K=\mathbb{C}$ or $\mathbb{Q}_p$ and $V$ be a $K$-vector space. A $V$-valued modular symbol is a map $F$ from the set $\Delta_0\times\Delta_0$ to $V$ such that for any $P,Q,R\in\Delta_0$,
$$
F(P,Q)=F(P,R)+F(R,Q).
$$
Denote by $\mathrm{Symb}\left(\Delta_0,V\right)$ the $K$-vector space of $V$-valued modular symbols.
\end{defn}
Notice that $\Gamma_0(N)$ acts by the left on $\mathrm{Symb}\left(\Delta_0,V\right)$. The $K$-vector space of $\Gamma_0(N)$-invariant modular symbols will be denoted by $\mathrm{Symb}\left(\Delta_0,V\right)^{\Gamma_0(N)}$.

Via Theorem \ref{key1} and the Manin continued fraction trick, one can constructively show
(cf.\,\cite{cremona}) that any closed path $\omega\in H_1\left(X\left(\Gamma_0(N)\right),\mathbb{Z}\right)$ can be expressed as a $\mathbb{Z}$-linear combination of paths of the form $\{g(0),g(i\infty)\}$, with $g\in\mathrm{SL}\left(2,\mathbb{Z}\right)$. These paths are called $M$-paths and they suffice to determine a modular symbol. The $K$-valued modular symbols supported on the $M$-paths are normally referred to as $M$-symbols.

The action of the Hecke algebra on $\mathrm{Symb}\left(\Delta_0,\mathbb{Q}_p\right)^{\Gamma_0(N)}$ is determined by the following double coset decomposition of the orbit space $\displaystyle\Gamma_0(N)\backslash\Gamma_0(N)\left(\begin{array}{ccc}p & \phantom{x} & 0\\0& \phantom{x} & 1\end{array}\right)\Gamma_0(N)$ (cf.\,\cite{cremona}, Proposition 1.6):
\begin{itemize}
\item[(i)]If $p\nmid N$, then
$$
\Gamma_0(N)\backslash\Gamma_0(N)\left(\begin{array}{ccc}p & \phantom{x} & 0\\0& \phantom{x} & 1\end{array}\right)\Gamma_0(N)=\displaystyle\bigcup_{u=0}^{p-1}\Gamma_0(N)\left(\begin{array}{ccc}1 & \phantom{x} & u\\0& \phantom{x} & p\end{array}\right)\Gamma_0(N)\bigcup\left(\begin{array}{ccc}p & \phantom{x} & 0\\0& \phantom{x} & 1\end{array}\right)\Gamma_0(N).
$$
\item[(ii)]If $p\vert N$, then
$$
\Gamma_0(N)\backslash\Gamma_0(N)\left(\begin{array}{ccc}p & \phantom{x} & 0\\0& \phantom{x} & 1\end{array}\right)\Gamma_0(N)=\displaystyle\bigcup_{u=0}^{p-1}\left(\begin{array}{ccc}1 & \phantom{x} & u\\0& \phantom{x} & p\end{array}\right)\Gamma_0(N).
$$
\end{itemize}
The action on paths is:
\begin{itemize}
\item[(i)] if $p\nmid N$, then
$$
\{g(0),g(i\infty)\}|_{T_p}=\sum_{u=0}^{p-1}\{(g(0)+u)p^{-1},(g(i\infty)+u)p^{-1}\}+\{pg(0), pg(i\infty)\}.
$$
\item[(ii)] If $p\vert N$, then
$$
\{g(0),g(i\infty)\}|_{T_p}=\sum_{u=0}^{p-1}\{(g(0)+u)p^{-1},(g(i\infty)+u)p^{-1}\}.
$$
\end{itemize}

Notice that, since $T_p$ acts on the rational numbers, it is again possible to decompose each path in the above sum as a $\mathbb{Z}$-linear combination of $M$-paths.
\begin{rem}If $K$ is a field, denote by $K^{alg}$ an algebraic closure of $K$. Fix compatible embeddings of $\mathbb{Q}$ in $\mathbb{Q}^{alg}$ and in $\mathbb{Q}_p^{alg}$. By Proposition \ref{finitude}, given $f\in S_2\left(\Gamma_0(N)\right)$, the classical modular integral $\phi_f$ can be seen both as a modular symbol with values in $\mathbb{C}$ or in a finite dimensional $\mathbb{Q}_p$-vector space.
\end{rem}

An important application of classical modular symbols is the explicit computation of spaces of modular forms. This topic is treated in detail in \cite{cremona}. In particular, one has the following result:
\begin{prop}[Pollack-Stevens \cite{pollack}] There exists an injection as Hecke-modules
$$
S_2\left(\Gamma_0(N)\right)\hookrightarrow\mathrm{Symb}\left(\Delta_0\right)^{\Gamma_0(N)}.
$$
\end{prop}
The method for computing modular forms consists in the determination of the eigenvalues of the Hecke operators acting on $\mathrm{Symb}\left(\Delta_0,\mathbb{C}\right)^{\Gamma_0(N)}$. These eigenvalues determine spaces of newforms for $\Gamma_0(N)$. The Manin continued fraction trick grants that it is easy to determine the eigenvalues by working on modular symbols. Hence, the Fourier expansion at infinity of such an eigenform has as $n$-th Fourier coefficient the computed eigenvalue for $T_n$.

\subsection{Quadratic modular symbols}
Let $\Gamma$ be an arithmetic Fuchsian group of the first kind commensurable with $\Gamma^1=\phi(\mathcal{O}_{H}^1)$ for some indefinite quaternion $\mathbb{Q}$-algebra $H$ and let $\tau\in\mathcal{H}$ be a quadratic imaginary point. If $\Gamma^1=\mathrm{SL}\left(2,\mathbb{Z}\right)$ and $\Gamma=\Gamma_0(N)$ we allow $\tau\in\mathbb{Q}\cup\{i\infty\}$.

Let $p$ be a prime. Choose a quaternion $\omega_p\in\mathcal{O}_H^1$ of reduced norm $p$ and set $\gamma_p=\phi(\omega_p)$. In \cite{shimura} it is proved that there exists $d\in\mathbb{N}$ such that $[\Gamma:\Gamma\cap\gamma_p\Gamma\gamma_p^{-1}]=d$. Hence, there is a coset decomposition
$$
\Gamma\gamma_p\Gamma=\displaystyle\bigcup_{a=1}^d\gamma_a\Gamma.
$$

Let us denote
$$
\Delta_0^{\tau,p}=\displaystyle\bigcup_{n=0}^{\infty}\bigcup_{\begin{array}{c}1\leq a_{i_j}\leq d\\1\leq j\leq n\end{array}}\gamma_{i_1}...\gamma_{i_n}\Gamma^1\tau
$$
and
$$
\Gamma(\Delta_0^{\tau,p}\times\Delta_0^{\tau,p})=\{(\gamma\sigma_1,\gamma\sigma_2);
\,\gamma\in\Gamma,\sigma_1,\sigma_2\in\Delta_0^{\tau,p}\}.
$$
\begin{defn}Let $K=\mathbb{C}$ or $\mathbb{Q}_p$ and $V$ a $K$-vector space. A $V$-valued $\tau$-modular symbol is a map $F$ from the set $\Gamma(\Delta_0^{\tau,p}\times\Delta_0^{\tau,p})$ to $V$ such that for any $P,Q,R\in\Delta_0^{\tau,p}$ and for any $\gamma\in\Gamma$,
$$
F(\gamma P,\gamma Q)=F(\gamma P,\gamma R)+F(\gamma R,\gamma Q).
$$
Denote by $\mathrm{Symb}\left(\Delta_0^{\tau,p},V\right)^{\Gamma}$ the $K$-vector space of $\tau$-modular symbols which are invariant under $\Gamma$.
If $\tau$ is quadratic imaginary we will refer to the elements of $\mathrm{Symb}\left(\Delta_0^{\tau,p},V\right)^{\Gamma}$ as quadratic modular symbols.
\end{defn}

\begin{rem}Notice that $\Gamma$ acts trivially on $\mathrm{Symb}\left(\Delta_0^{\tau,p},V\right)^{\Gamma}$. The motivation for our definition is that we want the matrices attached to the Hecke operator $T_p$ to act on $\mathrm{Symb}\left(\Delta_0^{\tau,p},V\right)^{\Gamma}$. Notice that for $\tau=i\infty$, for any prime $p$, for $\Gamma^1=\mathrm{SL}\left(2,\mathbb{Z}\right)$ and $\Gamma=\Gamma_0(N)$, $\Delta_0^{i\infty,p}=\Delta_0$ and $\Gamma(\Delta_0^{i\infty,p}\times\Delta_0^{i\infty,p})=\Delta_0\times\Delta_0$.
\end{rem}

Let us turn to the modular case. Let $\tau\in\mathcal{H}$ be a quadratic imaginary point satisfying $\tau^2+D=0$ with $D$ square free and $\tau$ non-elliptic. We have the following
\begin{prop}Let $p$ be a prime number such that $\left(\frac{-D}{p}\right)=-1$. There is an injection
$$
\mathrm{Symb}\left(\Delta_0,V\right)^{\Gamma_0(N)}\hookrightarrow \mathrm{Symb}\left(\Delta_0^{\tau,p},V\right)^{\Gamma_0(N)}.
$$
\end{prop}
\begin{proof}
Given $0\leq a\leq p^n-1$, let us denote $\gamma_{a,p^n}=\left(\begin{array}{ccl}1 & \phantom{x} & a\\0 & \phantom{x} & p^n\end{array}\right)$.
Define a map
$$
\begin{array}{ccl}
I: \mathrm{Symb}\left(\Delta_0,V\right)^{\Gamma_0(N)} & \to & \mathrm{Symb}\left(\Delta_0^{\tau,p},V\right)^{\Gamma_0(N)}\\
F & \mapsto & I(F),
\end{array}
$$
where
$$
I(F)(\gamma\gamma_{a,p^n}\gamma_1\tau,\gamma\gamma_{b,p^m}\gamma_2\tau)= F(\gamma\gamma_{a,p^n}\gamma_1\infty,\gamma\gamma_{a,p^n}\gamma_1\infty).
$$
First, we check that the map is well defined. It suffices to check that if
$$\sigma_1\gamma_{a,p^n}\gamma_1\tau=\sigma_2\gamma_{b,p^m}\gamma_2\tau,\,\mbox{ with } \sigma_1,\sigma_2\in\Gamma_0(N),\gamma_1,\gamma_2\in\mathrm{SL}\left(2,\mathbb{Z}\right),$$
then $\sigma_1\gamma_{a,p^n}\gamma_1\infty=\sigma_2\gamma_{b,p^m}\gamma_2\infty$.

To prove this fact, suppose that $\sigma_1\gamma_{a,p^n}\gamma_1\tau=\sigma_2\gamma_{b,p^m}\gamma_2\tau$. This would imply that $\gamma\tau=\tau$ with $\gamma=\gamma_2^{-1}\gamma_{b,p^{m}}^{-1}\sigma^{-1}_2\sigma_1\gamma_{a,p^n}\gamma_1\in\mathrm{GL}\left(2,\mathbb{Q}\right)$. Notice that $\eta=p^{m}\gamma\in\mathrm{GL}\left(2,\mathbb{Z}\right)$ also fixes $\tau$ and has determinant $p^{n+m}$. If $\eta=\left(\begin{array}{ccc}a & \phantom{x} & b\\c & \phantom{x} &d\end{array}\right)$, then $p^{n+m}$ would be represented by the binary quadratic form $X^2+DY^2$, with discriminant $-4D$, hence, $-4D$ would be a square modulo $p^{n-m}r^2$, in particular modulo $p$, hence, $\left(\frac{-D}{p}\right)=1$, which would be a contradiction.
If $n=m=0$, then $\gamma\in\mathrm{SL}\left(2,\mathbb{Z}\right)$. Since $\tau$ is non-elliptic, $\sigma_1\gamma_{a,p^n}\gamma_1=\sigma_2\gamma_{b,p^n}\gamma_2$. Linearity and injectivity are obvious.
\end{proof}
\begin{rem}It does not seem easy to prove that the Hecke algebra acts on $\mathrm{Symb}\left(\Delta_0^{\tau,p},V\right)^{\Gamma_0(N)}$ (the operator $T_p$ certainly does). But since
$$\mathrm{Symb}\left(\Delta_0,V\right)^{\Gamma_0(N)}\hookrightarrow\displaystyle\bigcap_{\left(\frac{-D}{p}\right)=-1}\mathrm{Symb}\left(\Delta_0^{\tau,p},V\right)^{\Gamma_0(N)},$$ we see that the Hecke algebra acts on the image of $\mathrm{Symb}\left(\Delta_0,V\right)^{\Gamma_0(N)}$ in the intersection.
\end{rem}
\section{$p$-adic $L$-functions}
Classical modular symbols are at the core of the definition of the cyclotomic $p$-adic $L$-function. Next we recall this construction, following Mazur-Tate-Teitelbaum (\cite{mtt}).

\subsection{Cyclotomic $p$-adic $L$-functions}

Let $f\in S_2\left(\Gamma_0(N)\right)$ be a normalized newform. It is well known that the field $K_f$ obtained by adjoining to $\mathbb{Q}$ the Fourier coefficients of $f$ is a totally real number field. Denote by $\mathcal{O}_f$ the ring of integers of $K_f$ and by $\mathcal{O}_{f,p}$ the completion of $\mathcal{O}_f$ at a prime over $p$. Let $a_p$ denote the eigenvalue corresponding to the Hecke operator $T_p$. The Hecke polynomial of $f$ at $p$ is $X^2-a_pX+p$. Fix an embedding of $\mathbb{Q}$ in $\mathbb{Q}^{alg}$ and in $\mathbb{Q}_p^{alg}$ and consider $\alpha_p$ a root of the Hecke polynomial such that $|\alpha|_p>p^{-1}$. The fact that $f$ is an eigenvalue for $T_p$ and that $\alpha_p$ is a root of the Hecke polynomial allows to define a $p$-adic distribution:

\begin{defn}[Mazur-Tate-Teitelbaum \cite{mtt}]The following function on compact open discs of $\mathbb{Z}_p^*$ extends in a unique way to a $p$-adic distribution.

\begin{itemize}
\item[(i)] If $p\nmid N$, $\displaystyle\mu_{f,p}=\left(a+p^n\mathbb{Z}_p\right)=\alpha_p^{-n}\left(\int_{\frac{a}{p^n}}^{\frac{a}{p^n}+i\infty}f-\alpha^{-1}\int_{\frac{pa}{p^n}}^{\frac{pa}{p^n}+i\infty}f\right).$
\item[(ii)] If $p|N$, $\displaystyle\mu_{f,p}=\left(a+p^n\mathbb{Z}_p\right)=a_p^{-n}\int_{\frac{a}{p^n}}^{\frac{a}{p^n}+i\infty}f.$
\end{itemize}
\end{defn}

If $\alpha_p\in\mathcal{O}_{f,p}^*$, then $\mu_{f,p}$ is uniformly bounded on compact-open subsets of $\mathbb{Z}_p$ and $f$ is said to be ordinary at $p$. Otherwise, $f$ is called supersingular at $p$.

Denote $\mathcal{X}=\mathrm{Hom}_{\mathrm{cont}}\left(\mathbb{Z}_p^*,\mathbb{C}_p^*\right)$ the group of $\mathbb{C}_p^*$-valued continuous characters of $\mathbb{Z}_p^*$.
\begin{defn}[Mazur-Tate-Teitelbaum \cite{mtt}] Given $f\in S_2\left(\Gamma_0(N)\right)$ and $\chi\in\mathcal{X}$, the $p$-adic $L$-function is the Mellin-Mazur transform of $\mu_{f,p}$,
$$
L_p(f;\chi)=\int_{\mathbb{Z}_p^*}\chi(x)d\mu_{f,p}(x).
$$
\end{defn}

The key property of this $p$-adic $L$-function is that it interpolates values of the complex $L$-function attached to $f$:

\begin{prop}[Mazur-Tate-Teitelbaum \cite{mtt}] Let $\chi\in\mathcal{X}$ be a character of conductor $p^n$, with $n\geq 0$ and $\tau(\chi)$ its Gauss sum. Then
$$
L_p(f;\chi)=\frac{p^n}{\alpha^n\tau(\chi)}L(f;\chi,1).
$$
\end{prop}

The $p$-adic $L$-function cannot be identically zero.
Indeed, by a theorem of Rohrlich \cite{rohrlich}, it is non-zero for an infinite family of characters of conductor a power of $p$. In addition, in the ordinary case, it is non-zero for an infinite family of characters of the form $x\mapsto x^n$. Moreover, in the ordinary case, $L_p(f;\chi)=0$ only for a finite number of characters $\chi$. This is true also in the supersingular case assuming that $L_p(f;\chi)\neq 0$ for some $p$-adic character of infinite order (\cite{blanco}).

\subsection{Coset representatives and $p$-adic $L$-functions}
To define the classical $p$-adic $L$-function, one considers the standard right coset decomposition of the operator $T_p$ given in subsection 3.1. Indeed, one fixes a way of assigning a coset representative to any compact open subset of the form $a+p\mathbb{Z}_p$ for $1\leq a\leq p-1$. Even if we fix the standard right coset decomposition, another assignment of right coset representatives yields a different $p$-adic measure. To illustrate this, let us consider a permutation $\sigma\in\mathrm{Bij}\left(\{1,...,p-1\}\right)$. Define $\sigma(0)=0$. Denote again by $\sigma$ the bijection
$$
\begin{array}{ccc}
\sigma:\mathbb{Z}_p^*\cap\{1,...,p^n-1\} & \to & \mathbb{Z}_p^*\cap\{1,...,p^n-1\}\\
\displaystyle\sum_{i=0}^{n-1}a_ip^i & \mapsto & \displaystyle\sum_{i=0}^{n-1}\sigma(a_i)p^i.
\end{array}
$$
Let $\mu$ be a $\mathbb{Q}_p$-valued $p$-adic distribution. Define
$$
\mu^{\sigma}\left(a+p^n\mathbb{Z}_p\right)=\mu\left(\sigma(a)+p^n\mathbb{Z}_p\right).
$$
\begin{prop}For any $p$-adic distribution $\mu$ and for any permutation $\sigma\in\mathrm{Bij}\left(\{1,...,p-1\}\right)$, $\mu^{\sigma}$ is a $p$-adic distribution.
\end{prop}
\begin{proof}First notice that for any $a=\displaystyle\sum_{i=0}^{n-1}a_ip^i\in\mathbb{Z}_p^*$, we have the decomposition
$$a+p^n\mathbb{Z}_p=\displaystyle\bigcup_{j=0}^{p-1}a+jp^n+p^{n+1}\mathbb{Z}_p.$$
Hence,
$$
\mu^{\sigma}\left(a+p^n\mathbb{Z}_p\right)=\mu\left(\sigma(a)+p^n\mathbb{Z}_p\right)=\sum_{j=0}^{p-1}\mu\left(\sigma(a)+jp^n+p^{n+1}\mathbb{Z}_p\right).
$$
But since $\sigma$ is a permutation of indices between $1$ and $p-1$, the right hand side equals
$$
\sum_{j=0}^{p-1}\mu\left(\sigma(a)+\sigma(j)p^n+p^{n+1}\mathbb{Z}_p\right)=\sum_{j=0}^{p-1}\mu^{\sigma}\left(a+jp^n+p^{n+1}\mathbb{Z}_p\right).
$$
\end{proof}
Let $\chi\in\mathcal{X}$ a continuous $p$-adic character. We can define
$$
\chi_{\sigma}(a)=\chi(a^{\sigma^{-1}}a^{-1}).
$$
It is not difficult to see that $\chi\in\mathcal{X}$. Denote by $L_p^{\sigma}(f;\chi)$ the Mellin-Mazur transform of $\mu^{\sigma}$ at $\chi$. We have the following result:
\begin{prop}For any $\chi\in\mathcal{X}$, $L_p^{\sigma}(f;\chi)=L_p(f;\chi\chi_{\sigma})$. In particular, $L_p^{\sigma}$ is not identically zero.
\end{prop}
\begin{proof}
By definition, $\int_{\mathbb{Z}_p^*}\chi(x)d\mu(x)=\displaystyle\lim\sum_{a}\chi(a)\mu^{\sigma}(a+p^n\mathbb{Z}_p)$ where the limit is taken as $\{a+p^n\mathbb{Z}_p\}$ runs over the partitions of $\mathbb{Z}_p^*$. Hence,
$$
L_p^{\sigma}(f;\chi)=\lim\sum_{a}\chi(\sigma(a))\chi(a\sigma(a^{-1}))\mu(a^{\sigma}+p^n\mathbb{Z}_p)=\lim\sum_a\chi(a)\chi_{\sigma}(a)\mu(a+p^n\mathbb{Z}_p).
$$
The right hand side equals $L_p(f;\chi\chi_{\sigma})$.
As for the fact that it does not vanish, notice that the conductor of $\chi_{\sigma}$ is a power of $p$ and apply the main result of \cite{rohrlich}.
\end{proof}
As we can see, the Mazur-Mellin transforms of $\mu^{\sigma}$ and $\mu$ are not equal. However, we have
\begin{cor}For any $\sigma\in\mathrm{Bij}\left(\{1,...,p-1\}\right)$, $L_p(f;1)=L_p^{\sigma}(f;1)$.
\end{cor}
\subsection{Quadratic $p$-adic $L$-functions}

In \cite{blancobayer}, we have constructed $p$-adic $L$-functions based on quadratic modular symbols. The construction is as follows.

For any $a\in\mathbb{Z}$ with $|a|\in\{1,...,p-1\}$, there exists a unique couple of integers $x,y\in\mathbb{Z}$ such that $ax-py=1$ and  $x\in\{0,...,p-1\}$. Denote
$$\gamma_{a,p}=\left(\begin{array}{ccc}a & \phantom{x} & y\\p & \phantom{x} & x\end{array}\right).$$
For any $u\in\mathbb{Z}$ such that $0\leq |u|\leq p-1$, define
$$\gamma_u=\left(\begin{array}{ccc}1 & \phantom{x} & u\\0 & \phantom{x} & p\end{array}\right).$$
For any $0\leq |a|<p^n$ coprime to $p$, consider the expansion
$$a=a_0+\sum_{i=1}^{p^{n-1}}u_ip^i.$$
Define $$\gamma_{a,p^n}=\gamma_{u_{n-1}}\gamma_{u_{n-2}}...\gamma_{u_1}\gamma_{a_0,p}.$$
If $n\geq 2$, denote $\gamma_{pa,p^n}=\gamma_{a,p^{n-1}}$ and for $n=1$, denote $\gamma_{a,1}=p\gamma_{a,p}$. Notice that $\gamma_{a+up^n,p^{n+1}}=\gamma_u\gamma_{a,p^n}$.
Let $f\in S_2\left(\Gamma_0(N)\right)$ an eigenform for $T_p$, $a_p$ the corresponding eigenvalue and $\alpha_p$ and admissible root of the Hecke polynomial. If $p||N$, the only choice is $\alpha_p=a_p$. We propose the following
\begin{defn}
For any $a\in\mathbb{Z}$ coprime to $p$ and for any $n\geq 0$, denote by $\delta_f(\gamma_{a,p^n}(i))=\phi_f(\gamma_{a,p^n}(i))-\phi_f(\gamma_{-a,p^n}(i))$. Let $n\geq 1$.
\begin{itemize}
\item[(a)] If $p\parallel N$, we define
$$
\mu_{\mathcal{Q}}\left(a+p^n\mathbb{Z}_p\right)=\dfrac{1}{a_p^n}\delta_f\left(\gamma_{a,p^n}(i)\right).
$$
\item[(b)] If $p\nmid N$, we define
$$
\mu_{\mathcal{Q}}\left(a+p^n\mathbb{Z}_p\right)=\dfrac{1}{\alpha_p^n}\left(\delta_f\left(\gamma_{a,p^n}(i)\right)-\alpha_p^{-1}\delta_f\left(\gamma_{a,p^{n-1}}(i)\right)\right).
$$
\end{itemize}
Define
$$\mu_{\mathcal{Q}}\left(\mathbb{Z}_p^*\right)=\sum_{a=0}^{p-1}\mu_{\mathcal{Q}}\left(a+p\mathbb{Z}_p\right)$$
and
$$\mu_{\mathcal{Q}}\left(p\mathbb{Z}_p\right)=0.$$
\end{defn}

The fact that $f$ is an eigenfunction for $T_p$ and that $\alpha_p$ is a root of the Hecke polynomial grants that $\mu_{\mathcal{Q}}$ is a $p$-adic distribution. This distribution takes values in the Banach $\mathbb{C}_p$-vector space
$$
c_{\infty}\left(\mathbb{C}_p\right)=\left\{x=(x_n)_{n\geq 1}\in\mathbb{C}_p^{\mathbb{N}};\,\,\exists C_x\mbox{ such that }|x_n|\leq C_x\mbox{ for any }n\geq 1\right\}.
$$

If $\alpha_p\in\mathcal{O}_{f,p}^*$, then all the $\mathcal{O}_{f,p}$-linear combinations of the integrals $\delta_f(\gamma_{a,p^n}(i))$ are uniformly bounded by 1. In this case, as usual, we say that $f$ is ordinary at $p$. Otherwise we say that $f$ is supersingular at $p$.

\begin{defn}Let $f\in S_2\left(\Gamma_0(N)\right)$ be ordinary at $p$. The quadratic $p$-adic $L$-function attached to $f$ is
$$
L_p(f;\chi)=\int_{\mathbb{Z}_p^*}\chi(x)d\mu_{\mathcal{Q}}(x),
$$
where $\chi\in\mathcal{X}$.
\end{defn}
Since, for $f$ ordinary at $p$, $\mu_{\mathcal{Q}}$ is a $p$-adic measure, we can integrate continuous functions against $\mu_{\mathcal{Q}}$, and the integral belongs to $c_{\infty}\left(\mathbb{C}_p\right)$. We are interested in integrating continuous $p$-adic characters.
Recall that for any $x\in\mathbb{Z}_p^*$ we can write
$$
x=\omega(x)\langle x\rangle,
$$
where $\omega(x)$ is the unique $p-1$-th root of unity in $\mathbb{Z}_p^*$ congruent to $x$ mod $p-1$. Given $s\in\mathbb{Z}_p$, let us consider the $p$-adic character:
$$
\chi_s(x)=\mathrm{exp}_p\left(\mathrm{log}_p\left(\langle x\rangle\right)\right),\;\;x\in\mathbb{Z}_p^*.
$$
Here, the function $\mathrm{exp}_p$ is the $p$-adic exponentiation, which $p$-adically holomorphic in the disc $D(0,p^{\frac{-1}{p-1}})$, and $\mathrm{log}_p$ is the $p$-adic logarithm, which is holomorphic in the disc $D(1,p^{\frac{-1}{p-1}})$ (for details cf.\,\cite{robert}).

It is not difficult to obtain the following expansion:
\begin{equation}
\chi_s(x)=\sum_{n=0}^{\infty}\frac{\mathrm{log}_p\left(\langle x\rangle\right)^n}{n!}.
\label{seriesexpansion}
\end{equation}

Given $s\in\mathbb{Z}_p$, let us denote
$$
L_p(f;s)=L_p(f;\chi_s).
$$

\begin{prop}The quadratic $p$-adic $L$-function $L_p(f;s)$ is $p$-adically holomorphic in the disc $D(0,p^{-1})$.
\end{prop}
\begin{proof}
For any integer $n\geq 1$, lt us denote by $\sigma_n$ the sum of its digits in base $p$. It is well known (cf.\,\cite{robert}) that $|n!|=p^{\frac{-n+\sigma_n}{p-1}}$. Thus, there exists $C>0$ such that
$$
\left|\frac{\mathrm{log}_p\left(\langle x\rangle\right)^n}{n!}\right|\leq C.
$$
Since $\mu_{\mathcal{Q}}$ is uniformly bounded on the compact subsets of $\mathbb{Z}_p$, say, by $M$, we have that
$$
\left|\int_{\mathbb{Z}_p^*}\frac{\mathrm{log}_p\left(\langle x\rangle\right)^n}{n!}d\mu_{\mathcal{Q}}(x)\right|\leq MC.
$$
Hence, for any $s\in D(0,p^{-1})$
$$
\int_{\mathbb{Z}_p^*}\langle x\rangle^sd\mu_{\mathcal{Q}}(x)=\sum_{n=0}^{\infty}s^n\int_{\mathbb{Z}_p^*}\frac{\mathrm{log}_p\left(\langle x\rangle\right)^n}{n!}d\mu_{\mathcal{Q}}(x).
$$
\end{proof}

Let $E$ be an elliptic curve defined over $\mathbb{Q}$ of conductor $N$. Let $f_E$ be its corresponding weight 2 newform. There is a modular parametrization
$$
\begin{array}{rccc}
\Psi_E:&\Gamma_0(N)\backslash\mathcal{H} & \longrightarrow & \mathbb{C}/\Lambda_E\\
&\tau & \mapsto & \displaystyle \int_{\infty}^{\tau}f_E.
\end{array}
$$

\begin{defn}The quadratic $p$-adic $L$-function attached to $E$ is
$$
L_p(E;s)=L_p(f_E;\chi_{s-1}),
$$
\end{defn}

Denote by $\phi_E=\left(\wp_{\Lambda_E},\wp_{\Lambda_E}'\right)$ the Weierstrass uniformization map. The following result is well known
\begin{thm}[Birch \cite{birch}] Let $K$ be a quadratic imaginary field. If $\tau\in K\cap\mathcal{H}$, then
$$
\phi_E\left(\Psi_E(\tau)\right)\in E\left(K^{ab}\right).
$$
Here $K^{ab}$ denotes the maximal abelian extension of $K$.
\label{Birch}
\end{thm}

The quadratic $p$-adic $L$-function satisfies the following
\begin{thm}Let $E$ be an elliptic curve defined over $\mathbb{Q}$ of conductor $N$ and $\alpha_p$ an admissible root of the Hecke polynomial of the corresponding newform $f_E$. Let $p$ be a prime number and $a_p$ the eigenvalue of $T_p$ corresponding to $f_E$. Suppose that $f_E$ is ordinary at $p$.
\begin{itemize}
\item[(i)] If $p\nmid N$ and $\alpha_p=1$, then, for any $a$ coprime to $p$ and for any $n\geq 0$,
$$
\phi_E\left(\mu_{\mathcal{Q}}\left(a+p^n\mathbb{Z}_p\right)\right)\in E\left(\mathbb{Q}(\tau)^{ab}\right).
$$
In particular,
$$
\phi_E\left(L_p(E;1)\right)\in E\left(\mathbb{Q}(\tau)^{ab}\right).
$$
\item[(i)] If $p\parallel N$, then, for any any $a$ coprime to $p$ and for any $n\geq 0$,
$$
\phi_E\left(a_p^n\mu_{\mathcal{Q}}\left(a+p^n\mathbb{Z}_p\right)\right)\in E\left(\mathbb{Q}(\tau)^{ab}\right).
$$
In particular,
$$
\phi_E\left(a_pL_p(E;1)\right)\in E\left(\mathbb{Q}(\tau)^{ab}\right).
$$
\end{itemize}
Here $\mu_{\mathcal{Q}}$ stands for the quadratic $p$-adic measure attached to $f_E$.
\end{thm}
\begin{proof}
Suppose that $p\parallel N$ (the other case is analogous). We have:
$$
\mu_{\mathcal{Q}}\left(a+p^n\mathbb{Z}_p\right)=a_p^{-n}\delta_f\left(\gamma_{a,p^n}(\tau)\right)=a_p^{-n}\left(\Psi_E\left(\gamma_{a,p^n}(\tau)\right)-\Psi_E\left(\gamma_{-a,p^n}(\tau)\right)\right).$$

The arguments of the modular parametrization belong to the quadratic imaginary field $\mathbb{Q}(\tau)$. Hence, by using Theorem \ref{Birch},
$$
\phi_E\left(\Psi_E\left(\gamma_{a,p^n}(\tau)\right)\right),\phi_E\left(\Psi_E\left(\gamma_{-a,p^n}(\tau)\right)\right)\in E\left(\mathbb{Q}(\tau)^{ab}\right).
$$
Since $\phi_E$ is an isomorphism of groups the result holds.
\end{proof}
\subsection{$p$-adic $L$-functions for certain Shimura curves}
Let $H$ be a quaternion $\mathbb{Q}$-algebra of discriminant $D>1$, $\mathcal{O}$ an Eichler order of level $N\geq 1$ and $\mathcal{O}^*$ the group of units of reduced norm 1 in $\mathcal{O}$. Denote by $\Gamma$ the group $\phi(\mathcal{O}^*)$. Let $p$ be a prime. The Hecke operator $T_p$ is defined as the double coset $\Gamma\eta_p\Gamma$ acting of $S_2\left(\Gamma\right)$, where $\eta_p=\phi(\omega_p)$ with $\omega_p$ a quaternion of reduced norm $p$ (see \cite{shimura}). We thank Professor Y.\,Yang for showing  us the following
\begin{prop} Let $p$ be a prime.
\begin{itemize}
\item[(i)] If $p\nmid ND$, then $[\Gamma:\Gamma\cap\gamma_p^{-1}\Gamma\gamma_p]=p+1$.
\item[(ii)] If $p\vert N$ and $p\nmid D$, then $[\Gamma:\Gamma\cap\gamma_p^{-1}\Gamma\gamma_p]=p$.
\item[(iii)] If $p\vert D$, then $[\Gamma:\Gamma\cap\gamma_p^{-1}\Gamma\gamma_p]=1$.
\end{itemize}
\end{prop}
\begin{proof}The number of coset representatives is finite in any case (see \cite{shimura}) and it is enough to count it locally, since $H$ splits at infinity. If $p\nmid ND$, the Eichler order $\mathcal{O}\otimes\mathbb{Z}_p$ is conjugated to the local Eichler order $\left(\begin{array}{ccl}\mathbb{Z}_p & & \mathbb{Z}_p\\\mathbb{Z}_p & & \mathbb{Z}_p\end{array}\right)$. The matrices $\left(\begin{array}{ccc}p & \phantom{x} & 0\\0 & \phantom{x} & 1\end{array}\right)$ and $\left(\begin{array}{ccc}1 & \phantom{x} & j\\0 & \phantom{x} & p\end{array}\right)$ with $0\leq j\leq p-1$ are a family of coset representatives.
If $p\vert N$ but $p\nmid D$, the Eichler order $\mathcal{O}\otimes\mathbb{Z}_p$ is conjugated to the local Eichler order $\left(\begin{array}{ccc}\mathbb{Z}_p & & \mathbb{Z}_p\\N\mathbb{Z}_p & & \mathbb{Z}_p\end{array}\right)$. The matrices $\left(\begin{array}{ccc}1 & \phantom{x} & j\\0 & \phantom{x} & p\end{array}\right)$ with $0\leq j\leq p-1$ are a family of coset representatives in this case. Finally, if $p\vert D$, the Eichler order $\mathcal{O}\otimes\mathbb{Z}_p$ is a discrete valuation ring and we can take as element of norm $p$ any uniformizer of the Eichler order, and any element of norm $p$ of this Eichler order is a uniformizer, which differs from the first in a unit, hence, there is only one class.
\end{proof}

Next we propose a definition of a $p$-adic $L$-function for the Shimura curve $X(D,N)$ provided that $p\vert N$ but $p\nmid D$. Before doing this, note that the decomposition in coset representatives is not unique (both for cocompact and non-cocompact Shimura curves). In the non-cocompact case, having fixed the standard decomposition of the Hecke double coset operator, one has to assign a coset representative to any compact-open subset $D\left(j,p^{-1}\right)$, $1\leq j\leq p-1$, and for different assignments, the corresponding $p$-adic $L$-functions are different. Hence, fix a family of coset representatives $\{\gamma_j\}$ for the orbit space $\Gamma\backslash\Gamma\eta_p\Gamma$. Let $f\in S_2\left(\Gamma\right)$ be a weight 2 automorphic form which is an eigenform for the operator $T_p$, which is defined as
$$
T_p(f)=\displaystyle\sum_{j=1}^pf|_2\gamma_j.
$$
Let $a$ be a natural number less than $p^n$. Write $a=\displaystyle\sum_{i=0}^{n-1}a_ip^i$ with $0\leq a_i\leq p-1$. Fix a permutation $\sigma$ of the set $\{1,...p-1\}$ and set $\sigma(0)=0$. Denote
$$
\delta(f,\tau,\sigma,a)=\phi_f^{\tau}\left(\gamma_{\sigma(a_{n-1})}\cdots\gamma_{\sigma(a_0)}(\tau)\right)-\phi_f^{\tau}\left(\gamma_{\sigma(p-a_{n-1})}\cdots\gamma_{\sigma(p-a_0)}(\tau)\right).
$$
\begin{defn} Let $\sigma$ be a bijection of the set $\{0,...,p-1\}$ with $\sigma(0)=0$ and let $f\in S_2\left(\Gamma\right)$ be an eigenform for $T_p$ with eigenvalue $a_p$. Let $\tau\in\mathcal{H}$ be a quadratic imaginary point. The quadratic $p$-adic distribution attached to $f$, $\tau$ and $\sigma$ is defined by
$$
\mu^{\sigma}_{\mathcal{Q}}\left(D\left(a,p^n\right)\right)=a_p^{-n}\delta(f,\tau,\sigma,a),
$$
$$
\mu^{\sigma}_{\mathcal{Q}}\left(\mathbb{Z}_p^*\right)=\displaystyle\sum_{a=1}^{p-1}\mu^{\sigma}_{\mathcal{Q}}\left(D\left(a,p^n\right)\right),
$$
$$
\mu^{\sigma}_{\mathcal{Q}}\left(p\mathbb{Z}_p\right)=0.
$$
\end{defn}
\begin{prop}The function $\mu_{\mathcal{Q}}$ extends, in a unique way, to a $c_{00}\left(\mathbb{C}_p\right)$-valued $p$-adic distribution.
\end{prop}
\begin{proof}
Notice that $\mu_{\mathcal{Q}}\left(a+p^n\mathbb{Z}_p\right)\in\Sigma_{f,n}\otimes\mathbb{Z}_p[a_p^{-1}]$. Since each compact open subset can be covered by a finite number of discs, $\mu_{\mathcal{Q}}$ takes values on $c_{00}\left(\mathbb{C}_p\right)$.

To prove the distribution property, denote by $\gamma_j(z)$ the action of $\gamma_j$ on $z$ by M\"{o}bius transformation. Since $f$ is an eigenform for $T_p$, we have
$$
a_p\phi_f^{\tau}\left(\gamma_{\sigma(a_{n-1})}\cdots\gamma_{\sigma(a_0)}(\tau)\right)=p^{-1}\sum_{j=1}^p\int_{\gamma_{\sigma(a_{n-1})}\cdots\gamma_{\sigma(a_0)}(\tau)}^{\tau}f\left(\gamma_j(z)\right)dz.
$$
After changing variables, and taking into account that $j$ is a mute index, we have
\begin{equation}
a_p\phi_f^{\tau}\left(\gamma_{\sigma(a_{n-1})}\cdots\gamma_{\sigma(a_0)}(\tau)\right)=\sum_{j=1}^p\phi_f^{\tau}\left(\gamma_{\sigma(j)}\gamma_{\sigma(a_{n-1})}\cdots\gamma_{\sigma(a_0)}(\tau)\right)+K,
\label{compatibility1}
\end{equation}
where
$$
K=-\sum_{j=1}^{p}\int_{\gamma_j(\tau)}^{\tau}f(z)dz.
$$
In addition
\begin{equation}
a_p\phi_f^{\tau}\left(\gamma_{\sigma(p-a_{n-1})}\cdots\gamma_{\sigma(p-a_0)}(\tau)\right)=\sum_{j=1}^p\phi_f^{\tau}\left(\gamma_{\sigma(p-j)}\gamma_{\sigma(p-a_{n-1})}\cdots\gamma_{\sigma(p-a_0)}(\tau)\right)+K.
\label{compatibility2}
\end{equation}
Now, by subtracting equation \ref{compatibility1} from \ref{compatibility2}, the result follows.
\end{proof}
To distinguish between ordinary and supersingular forms, we need the following result on algebraicity:

\begin{thm}[Shimura \cite{shimura}] Let $H$ be a quaternion $\mathbb{Q}$-algebra of discriminant $D$ and let $\mathcal{O}$ be an Eichler order of level $N$. Let $p$ be a prime. Then the eigenvalue of the Hecke operator $T_p$ acting on $S_2(\Gamma)$ is an algebraic integer.
\label{shimalg}
\end{thm}
Let $a_p$ be the eigenvalue of $T_p$ attached to $f$. By Theorem \ref{shimalg}, $a_p$ is an algebraic integer. We say that $f$ is ordinary at $p$ if $|a_p|=1$, otherwise we say that $f$ is supersingular at $p$.
\begin{defn}Denote by $\Sigma_{p-1}$ the symmetric group of order $p-1$. Let $f\in S_2\left(\Gamma(D,N)\right)$ be an eigenform for the Hecke operator $T_p$, for $p\nmid D$, $p\vert N$. Let $\tau\in\mathcal{H}$ be a quadratic imaginary point. If $f$ is ordinary at $p$, the $p$-adic $L$-function attached to $f$ and $\tau$ is
$$
L_p(f;\sigma,\chi)=\int_{\mathbb{Z}_p^*}\chi(x)d\mu^{\sigma}_{\mathcal{Q}}(x).
$$
\end{defn}

\end{document}